\documentclass[a4paper,xdvi,twoside,10pt,headsepline]{scrartcl}

\usepackage[T1]{fontenc}
\usepackage[latin1]{inputenc}
\usepackage{scrpage2}
\usepackage{graphicx}
\usepackage{multirow}
\usepackage{amsmath, amsthm, amsfonts, amssymb}
\usepackage{mathbbol}
\usepackage{stmaryrd}
\usepackage[arrow, matrix, curve, xdvi, dvips]{xy}
\usepackage{geometry}
\clearscrheadfoot
\automark[section]{section}
\cehead{\leftmark}
\cohead{\rightmark}
\ofoot[]{\pagemark}

\theoremstyle{plain}
\newtheorem{theorem}           {Theorem} [section]
\newtheorem{lemma}    [theorem]{Lemma}
\newtheorem{cor}      [theorem]{Corollary}
\newtheorem{prop}     [theorem]{Proposition}

\newtheorem*{thm2}             {Theorem \ref{thm2}}

\theoremstyle{definition}

\theoremstyle{remark}
\newtheorem{rem}      [theorem]{Remark}

\numberwithin{equation} {section}
\numberwithin{theorem}  {section}

\newcommand         {\catC}          {\ensuremath{\mathcal C}}                  
\newcommand					{\catT}					 {\ensuremath{{\bf{\mathcal{T}}}}}  					
  				% ...

\DeclareMathOperator{\rad}         {rad}
\DeclareMathOperator{\soc}         {soc}
\DeclareMathOperator{\Ob}          {Ob}
\DeclareMathOperator{\annl}        {ann_{\mathit{l}}}

\newcommand         {\N}             {\ensuremath{\mathbb{N}}}
\newcommand         {\Z}             {\ensuremath{\mathbb{Z}}}

\newcommand         {\id}            {\ensuremath{\mathrm{id}}}
\newcommand         {\Id}            {\ensuremath{\mathrm{Id}}}
\newcommand         {\ch}            {\ensuremath{\mathrm{char}}}

\newcommand					{\m}							{{\ensuremath{{{\mathfrak{m}}}}}}

\newcommand					{\catP}					 {\ensuremath{{\bf{\mathcal{P}}}}}
\newcommand					{\catPfg}				{\ensuremath{{\bf{\mathcal{P}}}^{\text{fg}}}}

\newcommand         {\HOM}      [2] {\ensuremath{{{Hom_R({#1},{#2})}}}}

\newcommand         {\ds}       [1] {\bigoplus\limits_{#1}R}
\newcommand         {\W}				[1] {\ensuremath{\mathcal{W}({#1})}}
\newcommand         {\Wi}				[1] {\ensuremath{\mathcal{W}_2({#1})}}
\newcommand         {\F}  			[1] {\ensuremath{\mathbb{F}_{#1}}}

\begin{document}

\title{Triangulated Structures for projective Modules}
\author{Boryana Dimitrova}
\date{\today}
\maketitle

\begin{abstract}
Abstract: We give a characterisation of those local not necessary commutative rings, for which the category of projective modules admits a triangulation with the identity as translation functor. By "admits a triangulation" we mean that the category can be given the structure of a triangulated category that satisfies the standard set of axioms including the octahedral axiom.
\end{abstract}

\section{Introduction}
Among categories of modules that admit triangulated structure, the simplest example is probably that of $k$-vector spaces for a field $k$. Here the translation functor is the identity and distinguished triangles are given by sequences of the form $$\xymatrix{{X}\ar^-{u}[r]&{Y}\ar^-{v}[r]&{Z}\ar^-{w}[r]&{X}},$$ such that $$\xymatrix{{X}\ar^-{u}[r]&{Y}\ar^-{v}[r]&{Z}\ar^-{w}[r]&{X}\ar[r]^-{u}&{Y}}$$ is an exact sequence of $k$-modules. An alternative description is to say that the distinguished triangles are precisely the contractible triangles, as defined by Neeman in \cite[1.3.5]{Neeman}. It is almost immediately clear that all the axioms are satisfied thanks to the fact that all $k$-modules are projective. This example can be generalized without much effort: we adopt the analogous definition of distinguished triangles to the category of $R$-modules for any ring $R$, such that every module is projective, and we indeed obtain a triangulated structure for this category with translation functor the identity. What kind of rings have this property? The answer can be given straightforward - all modules being projective is equivalent to all modules being injective. In particular all submodules of the module $R$ are direct summands, which implies that $R$ is semisimple. On the other hand, all modules over a semisimple ring are projective. The Artin-Wedderburn theorem gives us a precise characterization of the semisimple rings: They are of the form $R=R_1\times R_2\times\cdots\times R_n\,,$ where $R_i$ is a full matrix ring $M_{n_i}(d_i)$ over some skew field $d_i$ for every $i$ from $1$ to $n$.\\
Looking at this example it seems reasonable to try to give a triangulation for a category of projective modules over more general rings by just taking exact sequences of modules. It is in fact true that if such a category admits a triangulation then the distinguished triangles are exact as a sequence of modules, but one should be aware of the fact that conversely, this naive construction does not produce a triangulation in general. For a counterexample one can think of the category of projective $\Z$-modules: clearly there is no exact sequence of projective (thus free) $\Z$-modules of the required form $$\xymatrix{{\Z}\ar^-{\cdot 3}[r]&{\Z}\ar[r]&{?}\ar[r]&{\Z}\ar[r]^-{\cdot 3}&{\Z}},$$ starting with the map $\Z\stackrel{\cdot 3}{\rightarrow}\Z$, so the first axiom of a triangulated category can not be satisfied. This is just one of the problems that may occur in the above construction when not all modules over the given ring are projective. Actually, there are plenty of examples of rings for which the category of projective modules does not admit a triangulation at all. Indeed, a necessary condition for this is that the ring is Quasi-Frobenius. However, it is by no means sufficient, and so the question for a characterization of rings which allow a triangulation of their categories of projective modules arises. In \cite{HL} an answer is given in the case where the considered rings are commutative, and so our aim is to generalize this result to not necessarily commutative rings. We are though going to restrict our attention to local rings. Denote by $\catP$ and $\catPfg$ the categories of projective left $R$-modules and finitely generated projective left $R$-modules. Our main theorem is:

\begin{thm2}
Let $R$ be a local ring with maximal ideal $\m$. Then there exists a triangulation for $\catP$, or $\catPfg$ respectively, with $\Sigma=\Id$ if and only if $\m^2=0$, $\m = Rx = xR\,\,\,\text{for all}\,\, x \in \m \backslash \{0\}$ and in addition one of the following conditions holds:
\begin{enumerate}
\item $R$ is a skew field, i.e., $\m=\{0\}$
\item\label{thm2,2}$\m=2R$
\item\label{thm2,3}$\ch R=2$, i.e., $1+1=0$ in $R$, and for some $x \in \m \backslash \{0\}$ there is a nontrivial element $r_x\in R/\m$ such that
\begin{enumerate}
\item $\sigma_x(r_x) =r_x$ \,\,\,
\item $\sigma_x^3(t)=r_x^{-1}tr_x$\,\,\,\,for every $t\in R/\m$
\end{enumerate}
\end{enumerate}
\end{thm2}
\cehead[]{\leftmark}
\cohead[]{\rightmark}
Here $\sigma_x$ is a certain automorphism of the skew field $R/\m$ that depends on the chosen generator $x$ of $\m$ (Lemma \ref{Sigma}).
In the proof of these results we are going to construct an explicit triangulation for the different cases, using an idea from the article of F. Muro, S. Schwede, N. Strickland \cite{MSS}. Interesting questions that come up are whether there are different triangulations in the considered cases and - if so - how they look like. We will also give answers to these, for the most part.\\

{\bf Acknowledgements:}
\\
This article is a revised version of my diploma thesis, written at the University of Bonn. I would like to thank my advisor Stefan Schwede for the interestic topic and the encouragement throughout the way, as well as the German Academic Exchange Service for the financial support during the entire time of my studies.\\

Now, let us start with some
\section{Preliminaries}
We begin by fixing some notation and recollecting the needed terminology.\\
Let $\catC$ be an additive category and $\Sigma: \catC\rightarrow \catC$ an equivalence. A \emph{triangle} in $\catC$ (with respect to $\Sigma$) is a diagram in $\catC$ of the form $$\xymatrix{{X}\ar^-{u}[r]&{Y}\ar^-{v}[r]&{Z}\ar^-{w}[r]&{\Sigma X}}$$
such that the composites $vu$, $wv$ and $(\Sigma u)w$ are the zero morphisms. A \emph{morphism of triangles} is then a commutative diagram of the form
$$\xymatrix{{X}\ar^-{u}[r]\ar^-{f}[d]&{Y}\ar^-{v}[r]\ar^-{g}[d]&{Z}\ar^-{w}[r]\ar^-{h}[d]&{\Sigma X}\ar^-{\Sigma f}[d]\\{X'}\ar^-{u'}[r]&{Y'}\ar^-{v'}[r]&{Z'}\ar^-{w'}[r]&{\Sigma X'.}}$$
In this context we also have the notion of a \emph{mapping cone} which is analogous to the definition of the mapping cone of a chain map, namely, we define the mapping cone of the above morphism to be the triangle
$$\xymatrix@C=15mm{{Y \oplus X'}\ar^-{\left(\begin{smallmatrix}-v&0\\g&u'\end{smallmatrix}\right)}[r]&{Z \oplus Y'}\ar^-{\left(\begin{smallmatrix}-w&0\\h&v'\end{smallmatrix}\right)}[r]&{\Sigma X \oplus Z'}\ar^-{\left(\begin{smallmatrix}-\Sigma u&0\\ \Sigma f&w'\end{smallmatrix}\right)}[r]&{\Sigma Y \oplus \Sigma X'.}}$$
Now, a \emph{triangulated category} $\catT$ is an additive category together with a self\-equivalence $\Sigma: \catT\rightarrow \catT$ and a class of triangles (with respect to $\Sigma$), called distinguished, such that:\\
\\
\begin{tabular}{ccl}
A.1	&$\cdot$&The class of distinguished triangles is closed under isomorphisms of triangles.\\
&$\cdot$&The triangle $\xymatrix{{X}\ar^-{\id}[r]&{X}\ar[r]&{0}\ar[r]&{\Sigma X}}$ is distinguished for every $X \in \Ob (\catT)$.\\
&$\cdot$&For any morphism $u:X\rightarrow Y$ in $\catT$ there exists a distinguished triangle of the form\\ &&{$\xymatrix{{X}\ar^-{u}[r]&{Y}\ar^-{}[r]&{Z}\ar^-{}[r]&{\Sigma X}}$}.\\
\\
A.2 &$\cdot$& A triangle $\xymatrix{{X}\ar^-{u}[r]&{Y}\ar^-{v}[r]&{Z}\ar^-{w}[r]&{\Sigma X}}$ is distinguished iff its translate \\&& $\xymatrix{{Y}\ar^-{v}[r]&{Z}\ar^-{w}[r]&{\Sigma X}\ar^-{-\Sigma u}[r]&{\Sigma Y}}$ is.\\
\\
A.3&$\cdot$& For any commutative diagram with distinguished rows of the form\\ 
&&
\begin{minipage}[t]{60mm}
$\xymatrix{{X}\ar^-{u}[r]\ar^-{f}[d]&{Y}\ar^-{v}[r]\ar^-{g}[d]&{Z}\ar^-{w}[r]&{\Sigma X}\ar^-{\Sigma f}[d]\\{X'}\ar^-{u'}[r]&{Y'}\ar^-{v'}[r]&{Z'}\ar^-{w'}[r]&{\Sigma X',}}$
\end{minipage}
\begin{minipage}[t]{70mm}
there is a morphism $h:Z\rightarrow Z'$ completing the diagram to a morphism of triangles, in a way that the mapping cone is distinguished.
\end{minipage}
\end{tabular}
\\
\\
This set of axioms is equivalent to the common one including the octahedral axiom as shown by Neeman in \cite[1.4.6]{Neeman} and \cite[1.8]{Nee2}, and it will turn out to be more convenient for our purpose.\\
Further, two maps of triangles will be called \emph{homotopic} if they differ by a \emph{homotopy}
$$\xymatrix@=2cm{{X}\ar^-{u}[r]\ar@<1ex>[d]^-{f'}\ar@<-1ex>[d]_-{f}&{Y}\ar_-{\Theta}[dl]\ar^-{v}[r]\ar@<1ex>[d]^-{g'}\ar@<-1ex>[d]_-{g}&{Z}\ar_-{\Phi}[dl]\ar^-{w}[r]\ar@<1ex>[d]^-{h'}\ar@<-1ex>[d]_-{h}&{\Sigma X}\ar@<1ex>[d]^-{\Sigma f'}\ar@<-1ex>[d]_-{\Sigma f}\ar_-{\Psi}[dl]\\{X'}\ar^-{u'}[r]&{Y'}\ar^-{v'}[r]&{Z'}\ar^-{w'}[r]&{\Sigma X'.}}$$
This means there are morphisms $\Theta$, $\Phi$ and $\Psi$ as in the diagram such that
\begin{eqnarray*}
g-g'&=&\Phi v + u' \Theta\\
h-h'&=&\Psi w + v' \Phi\\
\Sigma(f-f')&=&\Sigma(\Theta u) + w'\Psi.
\end{eqnarray*}
As one should expect homotopic morphisms have isomorphic mapping cones. For a proof of this fact see \cite[1.3.3]{Neeman}. Another notion that we are going to make use of are the \emph{contractible triangles}. We call a triangle contractible if the identity map is homotopic to the zero map. A homotopy between both maps is called a \emph{nullhomotopy}. Later on we are going to use the facts that a contractible triangle in a triangulated category is always distinguished, and that a morphism having a contractible triangle as source or target is homotopic to the zero map (\cite{Neeman} Proposition 1.3.8 and Lemma 1.3.6).\\
Coming back to what we originally wanted to do, we set the following \underline{notation}: $\catP$ will denote the category of projective $R$-modules and $\catPfg$ the category of finitely generated projective $R$-modules, where $R$ is a (not necessarily commutative) ring. Later on in this section $R$ will be assumed to be a local ring. When we write $R$-modules we will mean left $R$-modules and by an ideal (without adjective) we will understand a two-sided ideal. As usual \HOM{P}{Q} denotes the set of all $R$-module homomorphisms between the $R$-modules $P$ and $Q$. As we mentioned in the introduction, Hovey and Lockridge gave a complete characterization of the commutative rings for which $\catP$ and $\catPfg$ admit a triangulation with $\Sigma=\Id$. Some of their preliminary results apply also in the general case and we are going to use them.\\
Recall that a ring is called \emph{Quasi-Frobenius} (short QF) if every injective $R$-module is projective or equivalently every projective $R$-module is injective. There are many other characterizations of QF rings, cf. \cite[15.1]{Lam2} and \cite[24.20]{Faith}. In \cite[3.4]{HL} it is shown that if $\catP$ admits a triangulation, then $R$ is a QF ring. This is an easy consequence of the fact that every distinguished triangle in $\catP$ (or also $\catPfg$) is exact, since the functor $\HOM{R}{-}$ converts a given distinguished triangle into an exact sequence of $R$-modules isomorphic to the original triangle, and for this reason every injective module can be embedded into a projective one, thus, is projective itself:
$$\xymatrix{
{P}\ar^-{f}[r]&{Q}\ar@{->>}[d]\ar[r]&{T}\ar[r]&{\Sigma {P}}\\
&{M}\ar@{^(->}[ur]
}$$
Here, $\xymatrix{{P}\ar^-{f}[r]&{Q}\ar@{->>}[r]&{M}\ar[r]&{0}}$ is a presentation of the injective module $M$ by projective modules.\\
A very similar argument \cite[3.5]{HL} provides that if $\catPfg$ admits a triangulation then $R$ is a left and right IF ring. Recall that $R$ is said to be a left respectively right IF ring if all injective left respectively right $R$-modules are flat. Adding the condition that $R$ is left or right noetherian implies that $R$ is even a QF ring (\cite{Colby}).\\
Now, as mentioned in the beginning we are going to direct our attention to local rings. The reason for this is that only in this case we have sufficient information about the structure of QF rings so that it is possible to extract enough information for a possible triangulation. In the commutative case an arbitrary QF ring is a product of local QF rings, and so, considering only the local case is not really a restriction. We should specify what we mean by a \emph{local} ring. $R$ is defined to be local if there is a unique maximal left ideal. This turns out to be the same as requiring a unique maximal right ideal (that indeed coinsides with the maximal left ideal) or also asking for $R/\rad R$ to be a skew field (see for example \cite{Lam}). Note, that this is in general not the same as having a maximal two-sided ideal. From now on $R$ will denote a local ring and in the light of this we introduce the following \underline{additional notation}:
\begin{center}
\begin{tabular}{cccl}
&&$\m$& the unique maximal ideal of $R$\\
&&$d:=R/\m $& the residue skew field of R\\
&&$[r]\in d$& the residue class of an element $r$ of $R$\\
&&$\tilde{t}\in R$& a lift of an element $t$ of $d$\\
&&$R^\times$& the units in $R$\\
&&$Z(R)$& the center of $R$\\
&&$Z(d)$& the center of $d$\\
\end{tabular}\end{center}
By a theorem of Kaplansky \cite[2]{Kaplansky} we know that the projective modules over a local ring are exactly the free modules, thus, we are going to work in categories of free respectively finitely generated free modules.

\section{Main results} 
Before beginning with the proof of the central theorem we have to do some preparations. Throughout, we are going to differentiate between $\catP$ and $\catPfg$ only if necessary. First, we deduce some restrictions on the structure of the rings $R$ that might occur, provided $\catP$ or $\catPfg$ is triangulated. The following proposition can also be found in \cite{HL}, but for our further work it is reasonable to recall the proof.
\begin{prop}\label{thmPID}
If the category of projective modules $\catP$ over a local ring $R$ or the category of finitely generated projective modules $\catPfg$ over a local left or right noetherian ring $R$ admits a triangulation with translation functor $\Sigma=\Id$, then the following properties hold for the maximal ideal $\m$ of $R$:
\begin{enumerate}
\item $\m^2 = \{0\}$
\item $\m = Rx = xR\,\,\,\text{for all}\,\, x \in \m \backslash \{0\}$
\end{enumerate}
\end{prop}
\begin{proof}
Without loss of generality $\m$ is not the zero ideal. As we saw, under the assumptions of the theorem the ring $R$ is QF, and further it is also assumed to be local. Hence by \cite[3.9]{Nicholson}, $\soc R_R= \soc_RR$ (write $\soc R$ from here on) is a simple left and right module. In other words, $\soc R$ is an ideal that is both simple as a left and as a right $R$-module. Clearly, $\soc R$ is generated by one element, and we have $\soc R = Rx = xR$ for any $x\in \soc R\backslash\{0\}$.\\
We want to prove that $\m$ equals $\soc R$. Let us fix some nontrivial, thus generating element $x$ of $\soc R$. Because of the triangulation there is a distinguished triangle starting with the map $\cdot x$ (right multiplication by the element $x$). This is of the form
$\xymatrix@!C=5mm{{R}\ar[r]^-{\cdot x}&{R}\ar[r]^-{f}&{P}\ar[r]^-{g}&{R},}$
where $P$ is some free module. Since the triangle is exact, we have a short exact sequence of $R$-modules:
$$\xymatrix{{0}\ar[r] &{R/Rx}\ar[r]&{P}\ar[r] &{\text{ker}(\cdot x)}\ar[r]&{0}}$$ 
Note that ker$(\cdot x)$ equals the left annihilator $\annl(xR)=\annl(\soc R)$ of the socle of $R$, and that is  precisely $\m$ (see for exemple \cite[8.11]{Nicholson}). We recall that since $R$ is left artinian and left noetherian (because of QF), all finitely generated left $R$-modules have finite length, so we conclude that $P$ also has finite length and we have:
$$\text{l}(P)=\text{l}(\m)+\text{l}(R/Rx)=\text{l}(\m)+\text{l}(R)-1\,\,\,\,(\star)$$
Since $\m$ is maximal, $\text{l}(R/\m)$ is simple and so $\text{l}(\m)+1=\text{l}(R)$. Further, the length of $P$ is a multiple of the length of $R$. Coming back to $(\star)$ the above considerations apply: $k\text{l}(R)=\text{l}(R)-1+\text{l}(R)-1=2\text{l}(R)-2$. Thus, $k$ equals $1$ and l$(R)=2$. Therefore, $0\subset\soc R\subset R$ is already a composition series for $R$ and $\soc R$ equals $\m$.
It is now an easy calculation that $\m^2$ is trivial.
\end{proof}
\begin{rem}\label{triangles thmPID}
We observe that under the assumptions of the proposition a distinguished triangle for the morphism $\cdot x$ for some nontrivial element $x$ of $\soc R=\m$ is given by 
$$\xymatrix{{R}\ar[r]^-{\cdot x}&{R}\ar[r]^-{\cdot s}&{R}\ar[r]^-{\cdot t}&{R}},$$
where $s$ and $t$ are elements of $R$. Since $x$ is nontrivial, $\cdot s$ and $\cdot t$ can not be isomorphisms. Because $x$ is not a unit, they can not be trivial, either. Hence, $s$ and $t$ are elements of $\m\backslash\{0\}$. As such we can represent them in the form $s=xr_1$ and $t=r_2x$ for some appropriate units $r_1$ and $r_2$. Then we have an isomorphism of triangles
$$\xymatrix{
{R}\ar[r]^-{\cdot x}\ar@{=}[d]&{R}\ar@{=}[d]\ar[r]^-{\cdot xr_1}&{R}\ar[d]^-{\cdot r_1^{-1}}\ar[r]^-{\cdot r_2x}&{R}\ar@{=}[d]\\
{R}\ar[r]^-{\cdot x}&{R}\ar[r]^-{\cdot x}&{R}\ar[r]^-{\cdot \bar{r}x}&{R},}$$ where $\bar{r}={r_1r_2}$ is also a unit.
\end{rem}
Another consequence of the proposition is the following  

\begin{lemma}\label{Sigma}
Every $x\in\m\backslash\{0\}$ defines an automorphism $\sigma_x:d\rightarrow d$ of the residue skew field $d=R/\m$, with the property  $x\widetilde{t}=\widetilde{\sigma_x(t)}x \text{ for all } t\in d.$ 
\end{lemma}
\begin{proof}
We define $\sigma_x$\, via the property $x\tilde{t}=\widetilde{\sigma_x(t)}x$: Let $t$ be an element of $d$. Since $Rx=xR$, for a lift $\tilde{t}$ of $t$ we have an element $s$ of $R$ with the property $sx=x\tilde{t}$. We set $\sigma_x(t)=[s]\in d$.\\This assignment is well defined: Two lifts $\tilde{t}$ and $\bar{t}$ differ by an element $m\in \m$, and since $xm=0$, we have $x\bar{t}=x(\tilde{t}+m)=x\tilde{t}$. If $\bar{s}$ also satisfies $x\tilde{t}=\bar{s}x$, then $\bar{s}x=sx$, i.e., $(\bar{s}-s)x=0$. Hence, $(\bar{s}-s)$ is in $\m$, and therefore $[\bar{s}]$ equals $[s]$.\\
Next, we have to show that $\sigma_x$ is a skew field homomorphism: For $t_1$ and $t_2$ in $d$, $\widetilde{t_1}+\widetilde{t_2}$ is a lift of $t_1+t_2$. If $s_1$ and $s_2$ are elements with $s_ix=x\widetilde{t_i}$ for $i=1,2$, then $(s_1+s_2)x=x(\widetilde{t_1}+\widetilde{t_2})$. Now $[s_1+s_2]=[s_1]+[s_2]$ implies we have a homomorphism of groups with respect to $+$. Similarly for multiplication, $\widetilde{t_1}\widetilde{t_2}$ is a lift of $t_1t_2$. We have  $x\widetilde{t_1}\widetilde{t_2}=s_1x\widetilde{t_2}=s_1s_2x$, and hence $\sigma_x(t_1t_2)=[s_1s_2]=[s_1][s_2]=\sigma_x(t_1)\sigma_x(t_2)$. It is clear that $\sigma_x(1)=1$.\\
Analogously, we can construct ${\sigma_x}^{-1}$ via the property $\tilde{t}x=x\widetilde{\sigma_x^{-1}(t)}$ for every $t\in d$.
\end{proof}
\begin{rem}\label{relation between sigma's}
One can ask what the connection between two such automorphisms $\sigma_x$ and $\sigma_y$ is. Since we can write $y=x\tilde{r}$ for an appropriate $r\in d$, we observe for any $t\in d$:
\begin{eqnarray*} \widetilde{\sigma_y(t)}y=y\tilde{t}&\Leftrightarrow&\widetilde{\sigma_y(t)}x\tilde{r}=x\tilde{r}\tilde{t}\\&\Leftrightarrow&\widetilde{\sigma_y(t)}\widetilde{\sigma_x(r)}x=x\tilde{r}\tilde{t}\\&\Leftrightarrow&\sigma_y(t)\sigma_x(r)=\sigma_x(rt)\end{eqnarray*} 
Thus, $\sigma_y(t)$ is equal to $\sigma_x(r)\sigma_x(t)(\sigma_x(r))^{-1}$. In case $\sigma_x(r)$ is in the center of $d$, we have an equality between $\sigma_x$ and $\sigma_y$. In particular, if $d$ is commutative all nontrivial elements of $\m$ define the same automorphism.
\end{rem}

Now we continue with the main result:

\begin{theorem}\label{thm2}
Let $R$ be a local ring that satisfies properties (i) and (ii) of Proposition \ref{thmPID} and $\m$ its maximal ideal. Then there exists a triangulation for $\catP$ (resp. $\catPfg$) with $\Sigma=\Id$ if and only if one of the following conditions hold:
\begin{enumerate}
\item $R$ is a skew field, i.e., $\m=\{0\}$
\item\label{thm2,2}$\m=2R$
\item\label{thm2,3}$\ch R=2$, i.e., $1+1=0$ in $R$, and for some $x \in \m \backslash \{0\}$ there is a nontrivial element $r_x\in R/\m$ such that
\begin{enumerate}
\item $\sigma_x(r_x) =r_x$ \,\,\,
\item $\sigma_x^3(t)=r_x^{-1}tr_x$\,\,\,\,for every $t\in R/\m$
\end{enumerate}
\end{enumerate}
We are going to call these three cases the "semisimple", the "mixed-characteristic" and the "equicharacteristic" case, respectively. 
\end{theorem}
\begin {proof}
Considering the unique map $\Z\rightarrow R$ we are often going to talk of elements of $\Z$ as elements of $R$. What is meant should be clear from the context. In the first part of the proof we show that $R$ should satisfy one of the upper conditions: \\
If $\m$ is trivial, it is clear that $R$ is a skew field. For $\m\neq\{0\}$ there are two possibilities: $\m$ is generated by some prime number $p\in\Z$, or $R$ contains a field. Indeed, if every prime number $p$ is a unit in $R$, then so is every other element of $\Z$, and therefore $\mathbb Q$ can be embedded in $R$. If $p=0$ in $R$ for some prime $p$, then there is a ringhomomorphism $\mathbb F_p\rightarrow R$ that is a monomorphism. If none of the above cases occurs, then there exists a prime $p$ with $p$ an element of $\m\backslash\{0\}$. Therefore $p$ is a generator of $\m$.\\
Before continuing we want to make the following note: a distinguished triangle in a triangulated category is isomorphic to the triangle that is produced by altering two of the maps by a sign. It is in general not true that changing only one of the signs produces a distinguished triangle, but in our case since the translation functor is the identity, this is also true as a consequence of the second axiom for triangulated categories. Remark \ref{triangles thmPID} allows us to consider the following commutative diagram with distinguished triangles as rows:
$$\xymatrix{
{R}\ar[r]^-{\cdot rx}\ar@{=}[d]&{R}\ar[r]^-{\cdot x}\ar@{=}[d]&{R}\ar[r]^-{\cdot x}\ar@{.>}^{\cdot t}[d]&{R}\ar@{=}[d]\\
{R}\ar[r]^-{\cdot rx}&{R}\ar[r]^-{\cdot x}&{R}\ar[r]^-{\cdot -x}&{R\,,}\\
}$$
By Axiom 4 there is a filler $\cdot t$ completing the above to a morphism of triangles. This implies that $t$ is a unit in $R$ representing the class of $1\in d$, and thus $x=-x$. In the case, that a field $\mathbb F_p$ can be embedded in $R$, we conclude from $2x=0$ that $2$ equals $0$ or in other words $\mathbb F_p=\mathbb F_2$. In the case that the maximal ideal is generated by a prime $p$, $2p=0$ forces $p$ to be $2$.\\
In both of the cases where $R\neq d$, we also have the fillers in the diagrams
$$\xymatrix{
{R}\ar[r]^-{\cdot x}\ar@{=}[d]&{R}\ar[r]^-{\cdot x}\ar@{=}[d]&{R}\ar[r]^-{\cdot rx}\ar@{.>}[d]&{R}\ar@{=}[d]\\
{R}\ar[r]^-{\cdot x}&{R}\ar[r]^-{\cdot rx}&{R}\ar[r]^-{\cdot x}&{R}\\
}\text{ }\text{ }\text{ }\text{ }\text{ }\text{ }\text{ }\text{ }\text{ }\text{ }\text{ }\text{ }\text{ }  
\xymatrix{
{R}\ar[r]^-{\cdot x}\ar[d]^-{\cdot\widetilde{\sigma_x(s)}}&{R}\ar[r]^-{\cdot x}\ar[d]^-{\cdot\tilde{s}}&{R}\ar[r]^-{\cdot rx}\ar@{.>}[d]&{R}\ar[d]^-{\cdot\widetilde{\sigma_x(s)}}\\
{R}\ar[r]^-{\cdot x}&{R}\ar[r]^-{\cdot x}&{R}\ar[r]^-{\cdot rx}&{R\,,}\\
}$$
Here, $x$ still denotes some generating element of $\m$, $s$ is some element of $d$ and $\tilde{s}$ a lift of $s$ in $R$. Call the filler in the left diagram $\cdot t$. We have $xt=rx$ and $tx=rx$, thus $\sigma_x([t])=[r]$ and  $[t]=[r]$. In other words, $\sigma_x([r])$ equals $[r]$. Denote the right filler by $\cdot \tilde {u}$, for $u\in d$. We have the following conditions: $\sigma_x(u)=s$ and $[r]\sigma_x^2(s)=u[r]$. Together they force $\sigma_x^3(u)=[r]^{-1}u[r]$, where $u$ equals $\sigma_x^{-1}(s)$. Since $s$ was an arbitrary element of $d$ and $\sigma_x^{-1}$ is bijective, we conclude that $\sigma_x^3(u)=[r]^{-1}u[r]$ for every $u\in d$. Hence, we are done with the first part of the proof.\\
Actually we showed more than was stated: we showed that under the assumption that a triangulation exists, in both the mixed-characteristic and the equicharacteristic case the properties $a)$ and $b)$ are satisfied for any $x\in\m\backslash\{0\}$ together with some appropriate $r_x$.\\
The second step is to construct a triangulation for each of the three cases. The category of (finite dimensional) vector spaces over a skew field has a unique triangulated structure analogous to vector spaces over a field, that we introduced at the beginning. The other two cases are more interesting. We are going to treat them together. We are going to use the construction from \cite{MSS}, but complete and modify it for our purpose. When possible, we are going to stick to their notation.\\First fix an $x$ and an $r_x$ with the properties $a)$ and $b)$. Note that in case $\m=(2)$, we can take $x=2$ and $r_x=1$ or in general any other $r_x$ that is in the center of $d$. We also have $x=-x$ in both cases. Define distinguished triangles to be triangles isomorphic to a direct sum (finite when working in $\catPfg$) of contractible triangles and triangles $\triangle_R$ of the form: 
$$\xymatrix{{\triangle_R: R}\ar[r]^-{\cdot x}&{R}\ar[r]^-{\cdot x}&{R}\ar[r]^-{\cdot \widetilde{r_x}x}&{R}}$$
For a free $R$-module $M$ we will denote by $\triangle_M$ the triangle
$$\xymatrix{{\triangle_M: M}\ar[r]^-{\cdot x}&{M}\ar[r]^-{\cdot x}&{M}\ar[r]^-{\cdot \widetilde{r_x}x}&{M},}$$
that is a direct sum of the triangles $\triangle_R$ (choice of a basis is involved!).\\
We have to verify the three axioms.

A.1: Distinguished triangles are closed under isomorphisms by definition and the triangle $$\xymatrix{{P}\ar[r]^-{\id}&{P}\ar[r]&{0}\ar[r]&{P}}\\$$ is  contractible for any object $P$.\\ 
Now we need the following lemma:

\begin{lemma}\label{zerlegung} In the above context, for any morphism $f:P\rightarrow Q$ in the category $\catP$ (resp. $\catPfg$), there is a commutative diagram
$$\xymatrix{
{P}\ar[r]^-{f}\ar[d]^-{\cong}&{Q}\ar[d]^-{\cong}\\
{M\oplus N\oplus V}\ar[r]^-{\left(\!\begin{smallmatrix}1&0&0\\0& \cdot x&0\\0&0&0 \end{smallmatrix}\!\right )}&{M\oplus N\oplus W\,,}\\
}$$ for appropriate free modules $M$, $N$, $V$ and $W$.
\end{lemma}
From this lemma it is clear that $f$ is a part of a distinguished triangle as required, since there is a triangle starting with $f$ that is isomorphic to the distinguished triangle
$$\xymatrix@C=1,2cm{{M\oplus N\oplus V}\ar[r]^-{\left(\!\begin{smallmatrix}1&0&0\\0& \cdot x&0\\0&0&0 \end{smallmatrix}\!\right )}&{M\oplus N\oplus W}\ar[r]^-{\left(\!\begin{smallmatrix}0&0&0\\0& \cdot x&0\\0&0&1 \end{smallmatrix}\!\right )}&{V\oplus N\oplus W}\ar[r]^-{\left(\!\begin{smallmatrix}0&0&0\\0& \cdot \widetilde{r_x}x&0\\1&0&0 \end{smallmatrix}\!\right )}&{M\oplus N\oplus V.}
}$$

\begin{proof}[Proof of Lemma \ref{zerlegung}]
We can use Zorn's lemma to determine the largest free submodule of $P$ such that the restriction of $f$ on it is trivial: Let $\mathfrak{W}$ denote the set of all free submodules with this property. As usual, a partial ordering is given by the inclusion of submodules and the empty set is an element of $\mathfrak{W}$. We have to show that for every totally ordered subset $\{W_i\}_{i\in I^\prime}$, the union $\bigcup_{i\in I^\prime}W_i$ is also in $\mathfrak{W}$, i.e. is a free $R$-module. Though a filtered colimit of free modules over an arbitrary ring need in general not be a free module, we are in the special situation where every inclusion of free modules splits (free$=$projective$=$injective), and the complementary module is as a direct summand of a free module a projective$=$free module itself. For a precise argument one needs transfinite induction since we have a possibly uncountable union. One shows that there is a compatible system of bases of the modules $W_i$. The successor case is exactly the fact discussed above. A basis of $W_i$ can be completed to a basis of $W_{i+1}$ since there is a free complement. For the limit case, assumed there are compatible bases for all the modules $W_i$ for $i<\alpha$ for a given limit ordinal $\alpha$, then a basis for $\bigcup_{i<\alpha}W_i$ is given by the union of the bases for $W_i$, where elements of different bases are being identified via the inclusions. Now Zorn's lemma gives us a maximal element, say $V$, which again has a free complement in $P$. Use the same method to split this complement into two modules $M$ and $N$, such that $M$ is maximal free with the property that $f$ is injective on it. Using the maximality of $M$ and $V$ one easily calculates that $f$ restricted on $N$ is (after an appropriate choice of a basis) just multiplication by the element $x$. The statement follows.
\end{proof}
Now back to the second axiom:\\
A.2: A triangle is contractible iff its translate is. The translate of $\triangle_R$ is  
$$\xymatrix{{R}\ar[r]^-{\cdot x}&{R}\ar[r]^-{\cdot \widetilde{r_x}x}&{R}\ar[r]^-{\cdot x}&{R}\\},$$ and it is isomorphic to the original one because of $\sigma_x(r_x)=r_x$:
$$\xymatrix{
{R}\ar[r]^-{\cdot x}\ar@{=}[d]&{R}\ar[r]^-{\cdot x}\ar@{=}[d]&{R}\ar[r]^-{\cdot \widetilde{r_x}x}\ar@{.>}^{\cdot \widetilde{r_x}}[d]&{R}\ar@{=}[d]\\
{R}\ar[r]^-{\cdot x}&{R}\ar[r]^-{\cdot \widetilde{r_x}x}&{R}\ar[r]^-{\cdot x}&{R}\\
}$$
On the other hand, $\triangle_R$ is the translate of the triangle  
$$\xymatrix{{R}\ar[r]^-{\cdot \widetilde{r_x}x}&{R}\ar[r]^-{\cdot x}&{R}\ar[r]^-{\cdot x}&{R}\\},$$ which is also isomorphic to $\triangle_R$ by the same argument.

A.3: We start by showing that any diagram of distinguished triangles of the form
$$\xymatrix{
{A}\ar[r]^-{f}\ar[d]^-{\alpha}&{B}\ar[r]^-{i}\ar[d]^-{\beta}&{C}\ar[r]^-{q}&{A}\ar[d]^-{\alpha}\\
{A^{\prime}}\ar[r]^-{f^{\prime}}&{B^{\prime}}\ar[r]^-{i^{\prime}}&{C^{\prime}}\ar[r]^-{q^{\prime}}&{A^{\prime}}\\
}$$
can be completed to a morphism of triangles. Assume first that one of the rows is contractible. If this is the upper row, we can copy this part in the proof of Theorem 1 of \cite{MSS} without any change. However, there is no possibility of using the duality functor mentioned there for the case where the lower triangle is contractible since we are also working in $\catP$, so we give an explicit construction:\\
Let $(\Theta,\Phi,\Psi)$ denote a nullhomotopy for the lower triangle:
$$\xymatrix{
{A}\ar[r]^-{f}\ar[d]^-{\alpha}&{B}\ar[r]^-{i}\ar[d]^-{\beta}&{C}\ar[r]^-{q}&{A}\ar[d]^-{\alpha}\\
{A^{\prime}}\ar@<0.4ex>[d]^-{1}\ar@<-0.4ex>[d]_-{0}\ar[r]^-{f^{\prime}}&{B^{\prime}}\ar@<0.4ex>[d]^-{1}\ar@<-0.4ex>[d]_-{0}\ar[dl]|-{\Theta}\ar[r]^-{i^{\prime}}&{C^{\prime}}\ar@<0.4ex>[d]^-{1}\ar@<-0.4ex>[d]_-{0}\ar[dl]|-{\Phi}\ar[r]^-{q^{\prime}}&{A^{\prime}}\ar@<0.4ex>[d]^-{1}\ar@<-0.4ex>[d]_-{0}\ar[dl]|-{\Psi}\\
{A^{\prime}}\ar[r]^-{f^{\prime}}&{B^{\prime}}\ar[r]^-{i^{\prime}}&{C^{\prime}}\ar[r]^-{q^{\prime}}&{A^{\prime}}
}$$
Since the rows are exact we can factorize $i^\prime\beta$ through the cokernel of $f$ (in the category of $R$-modules), and then use that $C^\prime$ is injective, to construct a map $\gamma^\prime$ such that $\gamma^\prime i=i^\prime \beta$:
$$\xymatrix@!0@=25pt{
{B}\ar[rr]^-{i}\ar[dd]_-{\beta}\ar@{->>}[dr]&&{C}\ar@{.>}[dd]^-{\gamma^\prime}\\
&{\text{coker}f}\ar@^{(->}[ur]\ar[dr]\\
{B^\prime}\ar[rr]^-{i^\prime}&&{C^\prime}\\
}$$
Now set $\gamma=\gamma^\prime+\Psi(\alpha q-q^\prime\gamma^\prime)$. We have
$$\begin{aligned}
\gamma i&=\gamma^\prime i+\Psi(\alpha qi-q^\prime\gamma^\prime i)=i^\prime\beta\\
q^\prime\gamma&=q^\prime\gamma^\prime+q^\prime\Psi(\alpha q-q^\prime\gamma^\prime)=q^\prime\gamma^\prime+(1-\Theta f^\prime)(\alpha q-q^\prime\gamma^\prime)=\\
&=q^\prime\gamma^\prime+\alpha q-q^\prime\gamma^\prime-\Theta f^\prime\alpha q+\Theta f^\prime q^\prime\gamma^\prime=\alpha q
\end{aligned}$$
Thus, $\gamma$ is the desired map.\\
If both triangles are isomorphic to direct sums $\triangle_R$, then, using the isomorphisms, we are searching a filler in 
$$\xymatrix{
{P}\ar[r]^-{\cdot x}\ar[d]^-{f}&{P}\ar[r]^-{\cdot x}\ar[d]^-{g}&{P}\ar[r]^-{\cdot \widetilde{r_x}x}\ar@{.>}[d]&{P}\ar[d]^-{f}\\
{Q}\ar[r]^-{\cdot x}&{Q}\ar[r]^-{\cdot x}&{Q}\ar[r]^-{\cdot \widetilde{r_x}x}&{Q\,,}\\
}$$
for appropriate $f$ and $g$. We can think of these maps as matrices with entries $g_{ij}$ and $f_{ij}$ respectively, in $R$. Because of the commutativity of the first square we have the relation $f_{ij}x=xg_{ij}$. We define the desired filler by setting $h_{ij}=\widetilde{\sigma_x^{-1}([g_{ij}])}$ for some lift (choose one) of $\sigma_x^{-1}([g_{ij}])$. This definition makes the middle square commute, and we only need to check that the right square commutes. Using the relation between $g_{ij}$ and $f_{ij}$ as well as the properties of $x$ and $r_x$ we calculate
$$\widetilde{r_x}xf_{ij}=\widetilde{r_x}\widetilde{\sigma_x([f_{ij}])}x=\widetilde{r_x}\widetilde{\sigma_x^{2}([g_{ij}])}x=\widetilde{r_x}\widetilde{\sigma_x^3([h_{ij}])}x=\widetilde{[h_{ij}]}\widetilde{r_x}x=h_{ij}\widetilde{r_x}x.$$
We want to make a few comments:
\begin{itemize} 
\item Note that we can alter the filler by adding $\mu(\cdot x) +(\cdot x)\nu$ for any $\mu,\nu:P\rightarrow Q$.
\item It is also worth remarking that if $f$ is an isomorphism, so are $g$ and $h$: Indeed, assumed $f$ is an isomorphism, applying Lemma \ref{zerlegung} on $g$ and using the equality $g(\cdot x)=(\cdot x)f$ one easily sees that $g$ is an isomorphism. The same argument provides that $h$ is also one.
\item Given the map $f$, one can always define $g$ and $h$. Thus, in the case where $f$ is an  isomorphism, we can complete it to an isomorphism of the distinguished triangles (of course this  applies only in this special case where the triangles are isomorphic to direct sums of $\triangle_R$).\\
\end{itemize} 
Now we are ready to proceed and show that the completion can be done such that the mapping cone is distinguished. We have the following situation:
$$\phi=\left(\!\begin{smallmatrix}\phi_{11}&\phi_{12}\\ \phi_{21}&\phi_{22}\end{smallmatrix}\!\right ):X\oplus T\rightarrow Y\oplus U\,,$$
where $X\oplus T$ and $Y\oplus U$ are distinguished triangles with $T$ and $U$ contractible triangles, $X$ and $Y$ isomorphic to direct sums of $\triangle_R$, and $\phi=(\alpha,\beta,\gamma)$ and each of the $\phi_{ij}=(\alpha_{ij},\beta_{ij},\gamma_{ij})$ morphisms of triangles. We should modify $\gamma$ in an appropriate way. Observe that $\phi_{12},\phi_{21}$ and $\phi_{22}$ are nullhomotopic since in each case the source or the  target is contractible. Putting the three homotopies together gives a homotopy from $\left(\!\begin{smallmatrix}\phi_{11}&\phi_{12}\\ \phi_{21}&\phi_{22}\end{smallmatrix}\!\right )$ to $\left(\!\begin{smallmatrix}\phi_{11}&0\\0&0\end{smallmatrix}\!\right)$, which means that the corresponding mapping cones are isomorphic. On the other hand, the cone of $\left(\!\begin{smallmatrix}\phi_{11}&0\\ 0&0\end{smallmatrix}\!\right)$ is itself isomorphic to the direct sum of the mapping cones of $\phi_{11}$, $U$ and the translate of $T$. Since the last two are contractible, our problem reduces to the case $\phi_{11}:X\rightarrow Y$ with $X$ and $Y$ as above. For suitable free modules $M$, $N$, $V$ and $W$ we can represent the map $\alpha_{11}$ as in Lemma \ref{zerlegung}:
$$\xymatrix{{X_1}\ar[r]^-{\alpha_{11}}\ar[d]^-{\cong}&{Y_1}\ar[d]^-{\cong}\\
{M\oplus N\oplus V}\ar[r]^-{\left(\!\begin{smallmatrix}1&0&0\\0& \cdot x&0\\0&0&0 \end{smallmatrix}\!\right )}&{M\oplus N\oplus W}\\
}$$
As we observed, the vertical isomorphisms can be extended to isomorphisms of exact triangles. Thus, we can restrict our investigation to the case:
$$\xymatrix{
{M\oplus N\oplus V}\ar[r]^-{\cdot x}\ar[d]^-{\left(\!\begin{smallmatrix}1&0&0\\0& \cdot x&0\\0&0&0 \end{smallmatrix}\!\right )}&{M\oplus N\oplus V}\ar[r]^-{\cdot x}\ar[d]^-{\bar{\beta}_{11}}&{M\oplus N\oplus V}\ar[r]^-{\cdot \widetilde{r_x}x}\ar[d]^-{\bar{\gamma}_{11}}&{M\oplus N\oplus V}\ar[d]^-{\left(\!\begin{smallmatrix}1&0&0\\0& \cdot x&0\\0&0&0 \end{smallmatrix}\!\right )}\\
{M\oplus N\oplus W}\ar[r]^-{\cdot x}&{M\oplus N\oplus W}\ar[r]^-{\cdot x}&{M\oplus N\oplus W}\ar[r]^-{\cdot \widetilde{r_x}x}&{M\oplus N\oplus W}\\
}$$
Let $\bar{\alpha}_{11}$ and $\delta $ denote the maps from $M\oplus N\oplus V$ to $M\oplus N\oplus W$ given by \linebreak $\bar{\alpha}_{11}=\left(\!\begin{smallmatrix}1&0&0\\0& \cdot x&0\\0&0&0 \end{smallmatrix}\!\right )$ and $\delta=\left(\!\begin{smallmatrix}0&0&0\\0&1&0\\0&0&0 \end{smallmatrix}\!\right )$. The morphism $\bar{\beta}_{11}$ is of the form $\bar{\alpha}_{11}+\Phi(\cdot x)$ for some \linebreak $\Phi:M\oplus N\oplus V\rightarrow M\oplus N\oplus W$ (cf.\ the relation between $f$ and $g$ in the construction of a filler for the case where the rows where direct sums of $\triangle_R$). Now, set $\bar{\gamma}_{11}$ to be $\bar{\beta}_{11}+\Phi(\cdot x)+(\cdot x)\delta+(\cdot x)\Phi=\bar{\alpha}_{11}+(\cdot x)\delta+(\cdot x)\Phi$. We claim that $(\delta,\Phi,0)$ is a homotopy from $(\bar{\alpha}_{11},\bar{\beta}_{11},\bar{\gamma}_{11})$ to the morphism $\zeta=(\epsilon,\epsilon,\epsilon)$, where $\epsilon=\left(\!\begin{smallmatrix}1&0&0\\0&0&0\\0&0&0 \end{smallmatrix}\!\right )$. The three equations for a homotopy hold:
$$\begin{aligned}
\bar{\beta}_{11}-\epsilon &=\Phi(\cdot x)+(\cdot x)\delta\\
\bar{\gamma}_{11}-\epsilon &=\bar{\alpha}_{11}+(\cdot x)\delta+(\cdot x)\Phi-\epsilon=(\cdot x)\Phi=0(\cdot\widetilde{r_x}x )+(\cdot x)\Phi\\
\bar{\alpha}_{11}-\epsilon &=\delta(\cdot x)=\delta(\cdot x)+(\cdot\widetilde{r_x}x )0
\end{aligned}$$
It only remains to observe that the mapping cone of $\zeta$ is isomorphic to the direct sum of $\triangle_N$, $\triangle_{W}$, the translate of $\triangle_N$, the translate of $\triangle_V$ with the mapping cone of the identity on $\triangle_M$, which can be easily checked to be contractible.
\end{proof}
From the two theorems in this section we conclude that $\catP$ is triangulated with translation functor the identity if and only if $R$ satisfies properties $(i)$ and $(ii)$ of Proposition \ref{thmPID} as well as one of the conditions $(i)$, $(ii)$ or $(iii)$ of the above theorem. The analogous statement holds for $\catPfg$ provided $R$ is left or right noetherian.\\

In the proof of the theorem we have given an explicit triangulation for the studied categories. A natural question is whether there are different triangulated structures. In the semisimple case the triangulation is  always unique. This is due to the fact that there every morphism in the category $\catP$ (resp. $\catPfg$) is up to isomorphism of the form $$\xymatrix{{U_1\oplus U_2}\ar[r]^-{\left(\!\begin{smallmatrix} 0&1\\ 0&0\end{smallmatrix}\!\right )}&{U_2\oplus V_1},}$$
and this is part of the triangle 
$$\xymatrix{
{U_1\oplus U_2}\ar[r]^-{\left(\!\begin{smallmatrix} 0&1\\ 0&0\end{smallmatrix}\!\right )}&{U_2\oplus V_1}\ar[r]^-{\left(\!\begin{smallmatrix} 0&0\\ 0&1\end{smallmatrix}\!\right )}&{U_1\oplus V_1}\ar[r]^-{\left(\!\begin{smallmatrix} 0&0\\ 1&0\end{smallmatrix}\!\right )}&{U_2\oplus U_1,}
}$$
which is contractible. Thus, this triangle is distinguished in an arbitrary triangulation. Since every two distinguished triangles starting with the same morphism are (non canonically) isomorphic, and the class of distinguished triangles is closed under isomorphisms, we see that the triangulation is unique.\\
Let us now consider the other two cases. We will denote the triangulation induced by the generating triangle
$$\xymatrix{{R}\ar[r]^-{\cdot x}&{R}\ar[r]^-{\cdot x}&{R}\ar[r]^-{\cdot \widetilde{r_x}x}&{R}}$$ 
by $\bigtriangleup^{r_x}$. In Theorem \ref{thm2} $(iii)$, in order to have a triangulation, we required the existence of an element $r_x$ of $d$ with $\sigma_x(r_x) =r_x$ and $\sigma_x^3(t)=r_x^{-1}tr_x$ for some generator $x$ of $\m$. On the other hand, as we saw in the first part of the proof of the theorem, in both the mixed-characteristic and the equicharacteristic case, a given triangulation defines for every generator $x\in\m\backslash\{0\}$ an element $r_x$ of $d$ that satisfies the above conditions. Clearly, the triangulation induced by $x$ and $r_x$ coincides with the original one. All these facts imply that to compare triangulations it is enough to fix one particular generator of the maximal ideal and observe those elements of $d\backslash\{0\}$, that together with this generator satisfy conditions $a)$ and $b)$ of Theorem \ref{thm2}. In the mixed characteristic case it is reasonable to take $x=2$. We have a map 
$$\begin{aligned}\left\{\text{triangulations of }\catP\,(\text{resp.}\,\catPfg) \right\}&\rightarrow Z(d)\\
{\bigtriangleup }^{r_x} &\mapsto r_x
\end{aligned}$$
where $Z(d)$ denotes the center of the residue field $d$. Indeed, as we mentioned in the proof of the theorem, every element $r_x\in d^\times$ that defines a triangulation together with $x=2$, should be in the center of $d$. In the equicharacteristic case we can also construct such a map. For this we fix a triangulation $\bigtriangleup^{r_x}$ and define
$$\begin{aligned}\left\{\text{triangulations of }\catP\,(\text{resp.}\,\catPfg) \right\}&\rightarrow {Z(d)^{\sigma_x}}.\\
{\bigtriangleup }^{r_x^\prime} &\mapsto r_x^{-1}r_x^\prime
\end{aligned}$$
Here ${Z(d)^{\sigma_x}}$ is the field of fixed points of $Z(d)$ under $\sigma_x$. There is to check that $r_x^{-1}r_x^\prime$ lies in $Z(d)^{\sigma_x}$. We use again conditions (a) and (b) of Theorem \ref{thm2}:
$$\begin{aligned}
&\sigma_x(r_x^{-1}r_x^\prime)=\sigma_x(r_x^{-1})\sigma_x(r_x^\prime)=r_x^{-1}r_x^\prime,&\\
&t=r_x\sigma_x^3(t)r_x^{-1}=r_x^\prime\sigma_x^3(t)(r_x^\prime)^{-1}\,\, &\text{for any}\, t\in d
\end{aligned}$$
The second equation can be written as $r_x^{-1}r_x^\prime\sigma_x^3(t)=\sigma_x^3(t)r_x^{-1}r_x^\prime$ which together with the fact that $\sigma_x$ is an automorphism implies that $r_x^{-1}r_x^\prime$ is a central element of $d$.\\
We are now ready to formulate the results about different triangulations in the mixed- characteristic and the equicharacteristic case:

\begin{cor}\label{bijection1}
Provided $\catP$ $(\text{resp.}\,\catPfg)$ can be triangulated, the above maps induce bijections 
\begin{enumerate}
\item mixed-characteristic case:
$$\left\{\text{triangulations of }\catP\,(\catPfg) \right\}\stackrel{1:1}{\leftrightarrow}\left\{\text{nontrivial elements of }Z(d)\right\}$$
\item equicharacteristic case:
$$\left\{\text{triangulations of }\catP\,(\catPfg) \right\}\stackrel{1:1}{\leftrightarrow}\left\{\text{nontrivial elements of } Z(d)^{\sigma_x}  \right\}$$
\end{enumerate}
\end{cor}
\begin{proof}
In the mixed-characteristic case we have taken $x=2$ as usual. Thus $\sigma_2^3$ is the identity and an element $r_2$ of $d\backslash\{0\}$ defines a triangulation via the generating triangle $$\xymatrix{{R}\ar[r]^-{\cdot 2}&{R}\ar[r]^-{\cdot 2}&{R}\ar[r]^-{\cdot 2\widetilde{r_2}}&{R}}$$
if and only if $r_2$ is in $Z(d)\backslash \{0\}$.\\
On the other hand, two different elements $r_2$ and $r'_2$ of $Z(d)\backslash\{0\}$ define different triangulations. Else there should exist the filler in the diagram
$$\xymatrix{
{R}\ar[r]^-{\cdot 2}\ar@{=}[d]&{R}\ar[r]^-{\cdot 2}\ar@{=}[d]&{R}\ar[r]^-{\cdot 2\widetilde{r_2}}\ar@{.>}[d]&{R}\ar@{=}[d]\\
{R}\ar[r]^-{\cdot 2}&{R}\ar[r]^-{\cdot 2}&{R}\ar[r]^-{\cdot 2\widetilde{r'_2}}&{R,}\\
}$$
which would force $r_2=r'_2$.\\
In the equicharacteristic case in order to have more than one triangulation we again necessarily need at least two different elements $r_x,r^\prime_x$ of $d\backslash\{0\}$ that satisfy the conditions $a)$ and $b)$ with respect to the fixed generator $x$ of $\m$. The same argument as above shows that the induced triangulations are different. Hence we only have to show that the given map is surjective. For this it is enough to see that for every element $z$ of $Z(d)^{\sigma_x}$, $r_xz$ defines a triangulation $\bigtriangleup^{r_xz}$. Indeed, $r_xz$ is a fixed point of $\sigma_x$ since $r_x$ and $z$ are, and moreover $$\sigma_x^3(t)=r_x^{-1}tr_x=z^{-1}r_x^{-1}tr_xz={(r_xz)}^{-1}t(r_xz)$$ holds because $z$ is in the center of $d$.
\end{proof}

We know now whether there are different triangulations or not, but the question remains, if, whenever they exist, these different triangulations are equivalent or even isomorphic? That is, is there an equivalence or even an isomorphism $F$ of the category $\catP$ (resp. $\catPfg$), that together with a natural transformation $\eta:F\Sigma\rightarrow \Sigma F$ maps distinguished triangles of one triangulation to distinguished triangles of another triangulation. To answer these questions in full generality we possibly need a full classification of the occurring rings. We will restrict our attention to determine isomorphic triangulations in a very strong sense: where the functor $F$ is the identity.\\

Let us first make some observations: The center of $R$, $Z(R)$ is a subring of $R$ that is itself local. Denote by $d'$ the residue field of $Z(R)$, then clearly the following diagram commutes:
$$\xymatrix{
{Z(R)}\ar@{->>}[d]\ar@^{(->}[rrr]&&&{R}\ar@{->>}[d]\\
{d'}\ar@^{(->}[r]&{Z(d)^{\sigma_x}}\ar@^{(->}[r]&{Z(d)}\ar@^{(->}[r]&{d}
}$$
We note that $d/Z(d)$ is a skew field extension, and $Z(d)/d'$ and $Z(d)^{\sigma_x}/d'$ are field extensions. 
Using the result of the previous corollary, we want to describe the isomorphic classes of triangulations in the sense discussed above. We claim that the identification is realized by the projection maps
$$\begin{aligned}
Z(d)^\times & \rightarrow & Z(d)^\times/d'^\times\\
(Z(d)^{\sigma_x})^\times & \rightarrow & (Z(d)^{\sigma_x})^\times/d'^\times
\end{aligned}$$
In other words:

\begin{cor}\label{bijection2}
For the equivalence classes of isomorphic triangulations with $F=\Id$ as discussed above, there are the following bijections:
\begin{enumerate}
\item mixed-characteristic case:
$$\begin{aligned}\left\{\text{equivalence classes of triangulations of }\catP\,(\catPfg) \right\} & \stackrel{1:1}{\leftrightarrow}\left\{\text{elements of }Z(d)^\times/d'^\times\right\}\\
\begin{bmatrix}{\bigtriangleup}^{r_x}\end{bmatrix} & \mapsto  [r_x]
\end{aligned}$$
\item equicharacteristic case:
$$\begin{aligned}\left\{\text{equivalence classes of triangulations of }\catP\,(\catPfg) \right\}&\stackrel{1:1}{\leftrightarrow}\left\{\text{elements of }(Z(d)^{\sigma_x})^\times/d'^\times\right\}\\
\begin{bmatrix}{\bigtriangleup}^{r'_x}\end{bmatrix} & \mapsto  [r_x^{-1}r'_x]
\end{aligned}$$
where ${\bigtriangleup}^{r_x}$ is a fixed triangulation.
\end{enumerate}
\end{cor}
\begin{proof}
At the beginning we treat $(i)$ and $(ii)$ together.We claim that two nontrivial elements $r_x$ and $r'_x$ of $d$ define, together with the generator $x$ of $\m$, equivalent triangulations if and only if there is a lift of $r_x^{-1}r'_x$ in $R$ that is in $Z(R)$.\\
Assume first $r\in Z(R)$ is a lift of $r_x^{-1}r'_x$. To define a natural transformation $\eta$ choose for every $P\in\catP$ an isomorphism $f_P:P\rightarrow\bigoplus_{I_P}R$ for an appropriate set $I_P$. We insist that in case $P=\bigoplus_{I_P}R$, $f_P$ is the identity. Set $\eta_P:P\rightarrow P$ to be the composite
$$\xymatrix{
{P_{_{_{_{_{_{_{_{}}}}}}}}}\ar@<4pt>[r]^-{f_P}&{\ds{I_P}}\ar@<4pt>[r]^-{\cdot r}&{\ds{I_P}}\ar@<4pt>[r]^-{f_P^{-1}}&{P._{_{_{_{_{_{_{_{}}}}}}}}}
}$$
Note that $\eta_P$ is an isomorphism.\\
Since $r$ is an element of the center of $R$, this gives us a natural transformation $\eta:\Id\circ\Sigma=\Id\rightarrow\Id=\Sigma\circ\Id$. The triangle $$\xymatrix{{R}\ar[r]^-{\cdot x}&{R}\ar[r]^-{\cdot x}&{R}\ar[r]^-{\cdot \widetilde{r_x}x}&{R}}$$ is mapped under $(\Id,\eta)$ to the generating triangle $$\xymatrix{{R}\ar[r]^-{\cdot x}&{R}\ar[r]^-{\cdot x}&{R}\ar[r]^-{\cdot \widetilde{r^\prime_x}x}&{R}}.$$ 
For contractible triangles we have: if $(\Theta,\Phi,\Psi)$ is a nullhomotopy for $$\xymatrix{{P}\ar[r]^-{u}&{Q}\ar[r]^-{v}&{T}\ar[r]^-{w}&{P}},$$ then $(\Theta,\Phi,\Psi\eta_P^{-1})$ is one for $$\xymatrix{{P}\ar[r]^-{u}&{Q}\ar[r]^-{v}&{T}\ar[r]^-{{\eta}_{_P} w}&{P}}.$$
Clearly, direct sums of contractible triangles with triangles isomorphic to sums of generating triangles are also mapped to distinguished triangles under $(\Id,\eta)$. Therefore $(\Id,\eta):(\catP,{\bigtriangleup}^{r_x})\rightarrow(\catP,{\bigtriangleup}^{r_x^\prime})$ is an isomorphism of triangulated categories.\\
Conversely, if a natural transformation $\eta:\Id\rightarrow\Id$ is given, it defines an element $r$ in $Z(R)$ via the map $\eta_R=\cdot r$. If 
$(\Id,\eta):(\catP,{\bigtriangleup}^{r_x})\rightarrow(\catP,{\bigtriangleup}^{r_x^\prime})$ is an isomorphism of triangulated categories, then the generating triangle of ${\bigtriangleup}^{r_x}$ should be mapped to the generating triangle of ${\bigtriangleup}^{r_x'}$. Hence, we see that $r_x[r]$ equals $r'_x$. This forces $[r]=r_x^{-1}r'_x$ and so $r$ is a lift of $r_x^{-1}r'_x$ in the center of $R$.\\
Translating what we showed into other language: in the light of Corollary \ref{bijection1}, the elements of $Z(d)^{\sigma_x}$ that are in the same equivalence class as $r_x$ are precisely those of the form $r_xz$ for $z\in(d')^\times$. This yields the result for the equicharacteristic case. Taking $x=2$ and using $\sigma_2=\id$ gives the one for the mixed-characteristic case.
\end{proof}

A direct result from this corollary is that if the ring $R$ is commutative, then all the triangulations are isomorphic.\\

\section{Examples}
In this section we want to give examples for the rings occurring in Theorem \ref{thm2}. In particular, we are interested in the existence of  non commutative ones. For the semisimple case there are clearly plenty of examples - for every skew field the categories of vector spaces and finitely generated vector spaces are triangulated. We therefore start with the mixed-characteristic case.\\
Recall the ring of Witt vectors $\W{k}$ over a perfect field $k$ of characteristic $p$. It is a commutative discrete valuation ring with maximal ideal $\hat{\m}=p\W{k}$ and residue field $k$. The ring of Witt vectors of length $2$ then is defined as $\Wi{k}:=\W{k}/\hat{\m}^2$, a local ring with only one nontrivial ideal, namely $\bar{\m}=p\Wi{k}$. Another description of $\Wi{k}$ is in terms of vectors $(a_0,a_1)\in k\times k$ where the addition and the multiplication are given by: 
$$\begin{aligned}
(a_0,a_1)+(b_0,b_1)&=(a_0+b_0,a_1+b_1+\dfrac{1}{p}\left ( a_0^p+b_0^p-\left ( a_0+b_0\right )^p\right )\\
(a_0,a_1)\cdot(b_0,b_1)&=(a_0\cdot b_0,{b_0}^pa_1+b_1{a_0}^p)
\end{aligned}$$
For example, for $p=2$ this reduces to 
$$\begin{aligned}
(a_0,a_1)+(b_0,b_1)&=(a_0+b_0,a_1+b_1+{a_0}{b_0})\\
(a_0,a_1)\cdot(b_0,b_1)&=(a_0\cdot b_0,{b_0}^2a_1+b_1{a_0}^2).
\end{aligned}$$
Componentwise application of the Frobenius automorphism
$$\begin{aligned}
Fr:k & \rightarrow k\\
   a & \mapsto a^p
\end{aligned}$$
induces a ring homomorphism 
$$\begin{aligned}
F:\Wi{k}&\rightarrow\Wi{k}\\
(a_0,a_1)&\mapsto ({a_0}^p,{a_1}^p)
\end{aligned}$$ 
that is indeed an isomorphism since $k$ is perfect. If we now take $k$ to be of characteristic $2$, then we immediately obtain a commutative example for (ii). For a non commutative one we take $k$ of characteristic $2$, but different from $\F{2}$, and set $R=\Wi{k}((X,{F}))$, the skew Laurent series ring associated to $\Wi{k}$ and ${F}$. Elements in $R$ are of the form $a=\sum_{i=-n}^{\infty}{a_iX^{i}}$, where $n\in\N$ depends on $a$ and the $a_i$'s are elements of $\Wi{k}$. Addition is done componentwise, multiplication is given by the formula
$$\sum\limits_{i=-n}^{\infty}{a_iX^{i}}\cdot\sum\limits_{i=-m}^{\infty}{b_iX^{i}}=\sum\limits_{k=-n-m}^{\infty}{\Bigl(\sum\limits_{i+j=k}{a_i{F}^i(b_j)}\Bigr)X^{k}}.$$
In other words, we take the usual multiplication in the Laurent series ring except for the fact that elements of $\Wi{k}$ do not commute with $X$, but are multiplied by the rule $Xa_i={F}(a_i)X$. One easily sees that this makes $R$ into a unital non commutative ring. We claim that $R$ is local with maximal ideal generated by $2$: For this purpose we show that $a=\sum_{i=-n}^{\infty}{a_iX^{i}}\in R$ is a unit if and only if one of the $a_i$'s is a unit in $\Wi{k}$. One direction is clear: If all coefficients are non units, i.e., are in the maximal ideal $\bar{m}$ of $\Wi{k}$, we have 
$$\sum\limits_{i=-n}^{\infty}{a_iX^{i}}\cdot\sum\limits_{i=-n}^{\infty}{a_iX^{i}}=\sum\limits_{k=-2n}^{\infty}{\Bigl(\sum\limits_{i+j=k}{a_i{F}^i(a_j)}\Bigr)X^{k}}=0\,,$$
since $F$ maps the maximal ideal $\bar{\m}$ to itself, and therefore for every $k\in\N$ holds \linebreak $\sum_{i+j=k}{a_i{F}^{i}(a_j)}\in{\bar{\m}^2}=\{0\}$. Hence, $a$ cannot be a unit. Now suppose $a_{-n}$ is a unit. We can inductively define the coefficients $b_j$ of $b=\sum_{j=n}^{\infty}{b_jX^j}$ in a way such that $b$ is a right inverse for $a$, and in the same manner one can find a left inverse for $a$. We conclude that $a$ is a unit. If $j>-n$ is minimal with $a_j\in\Wi{k}^\times$, set  $c=\sum_{i=-n}^{j-1}{a_iX^{i}}$ and $d=\sum_{i\geq j}{a_iX^{i}}$. So $a$ equals $c+d$ with $c^2=0$ and $d$ a unit. Note that $(uc)^2$ equals $0$ for any $u\in\Wi{k}((X,F))$ since the coefficients of $uc$ are still in $\bar{\m}$. We have:
$$\begin{aligned}
a(d^{-1}c-1)&=(c+d)(d^{-1}c-1)=\\
&=d(d^{-1}c+1)(d^{-1}c-1)=\\
&=d((d^{-1}c)^2-d^{-1}c+d^{-1}c-1)=\\
&=-d
\end{aligned}$$
Therefore, $a$ is right invertible. Similarly, one shows that it is also left invertible, and so the proof of the claim is done.\\
Now $\bar{\m}=2\Wi{k}$ implies, that the non units of $R$ form a right ideal, namely $\m=2R$ which is of course also a left ideal since $2$ is central, so we conclude that $R$ is local. Finally, we note that we already saw that $\m^2$ is trivial. It is also clear that there are no other nontrivial left or right ideals since this was the case by $\Wi{k}$. Therefore, $\m=xR=Rx$ for any element $x$ of $\m\backslash\{0\}$.\\

In the equicharacteristic case we use a construction similar to the skew Laurent series: the skew polynomials. For a ring $R^\prime$ and a ring homomorphism $\tau:R^\prime\rightarrow R^\prime$, the associated skew polynomials $R^\prime[X,\tau]$ differ from the usual polynomials $R^\prime[X]$ in the definition of the multiplication: It is given, as in the case of the skew Laurent series, by
$$\sum\limits_{i=0}^{n}{a_iX^{i}}\cdot\sum\limits_{i=0}^{m}{b_iX^{i}}=\sum\limits_{k=0}^{n+m}{\Bigl(\sum\limits_{i+j=k}{a_i\tau^i(b_j)}\Bigr)X^{k}}.$$
Take $R^\prime$ to be the field $\F{8}$ and $\tau$ again the Frobenius automorphism. Then we have $\tau^3=\id$. Now set $R=\F{8}[X,\tau]/(X^2)$, where $(X^2)$ denotes the ideal generated by $X^2$. Similar to the previous example one sees that precisely the elements of the form $a_0+a_1X$ mod $(X^2)$, with $a_0\neq 0$, are the units of $R$. Therefore, there is a unique ideal $\m$ in $R$, that is generated by the class of $X$: $\m=R[X]=([X])=[X]R$. Using the notation from Theorem \ref{thm2}, for $x=[X]$ we have $\sigma_x=\tau:\F{8}\rightarrow\F{8}$. Further, $\sigma_x^3$ is the identity map, and one can take $r_x=1$, so all the required conditions are satisfied. Since in this particular example $0$ and $1$ are the only fixed points of $\sigma_x$, we have uniqueness of the triangulation.\\If we exchange $\F{8}$ by $\F{2^6}$ and $\tau$ by $\tau^2$ we obtain another example, this time though, the fixed points of $\sigma_x=\tau^2$ are exactly the elements of $\F{4}$ understood as elements of $\F{2^6}$ via the canonical inclusion. Therefore, we can take $r_x$ to be any element of $\F{4}^\times$. Each of them defines a different triangulation but all these triangulations are isomorphic by Corollary \ref{bijection2}, since every such $r_x$ has a lift in the center of $R$ given by the class $[r_x+0\cdot X]$.\\

Looking back one can ask whether there are different triangulations in the first example. The condition on an element $r\in d\backslash\{0\}$ for defining a triangulation is $r$ to be in the center of $$d=\Wi{k}((X,{F}))/2\Wi{k}((X,{F}))\cong k((X,Fr)).$$
Using the above isomorphism we view elements of $d$ as elements in $k((X,Fr))$. One can calculate that in the case $Fr^i\neq\id$ for every $i\in\Z\backslash\{0\}$ the center consists only of the elements $0$ and $1$. If there exists $q\in\N$ with the property $Fr^q=\id$, then the elements in the center of $d$ are of the form $\sum_{i=-n}^{\infty}a_iX^{iq}$ with $n\in\N$, $a_i\in\{0,1\}$ and $q$ is minimal with the required property. Corollary \ref{bijection1} implies we have different triangulations in this second case, but for every $\sum_{i=-n}^{\infty}a_iX^{iq}$ in $Z(d)^\times$ we can find a lift $\sum_{i=-n}^{\infty}\bar{a}_iX^{iq}$ in the center of $R$. Here, for $a\in k=\Wi{k}/2\Wi{k}$ we take $\bar{a}=(a,0)$. Thus, also here all the triangulations are isomorphic.

\section*{Remarks}
In this article we have treated a special kind of rings, namely local ones. In the commutative case this is not a major restriction since commutative QF rings are products of local QF rings, and so, using the fact that for $R\cong R'\times R''$ holds $\text{Mod}R\cong\text{Mod}R'\times \text{Mod}R''$, one can characterize all the rings $R$ for which the category of projective modules admits a triangulation. The situation of non commutative rings is different. Surely we can also generalize our result for products of local rings. Using the Morita equivalence between a ring $R$ and the full matrix ring $\text{M}_n(R)$ for $n\in\N$ one can also generalize the result for a product of matrix rings over local rings. However, this still doesn't cover all the possibilities of how a QF ring could look like in the general case, so the question if a complete characterization can be done, remains open.\\
At the end we want to remark that our mixed-characteristic case provides us with some additional examples of the so called \emph{exotic} triangulated categories introduced by Muro, Schwede and Strickland in \cite{MSS}. These are categories that are neither algebraic nor topological triangulated (see \cite{Keller06} and \cite{Schwede}). They are characterized by the fact that every object in the category is, as named in \cite{MSS}, \emph{exotic}. This means that for every object $M$ in the category, there is an exact triangle of the form
$$\xymatrix{{M}\ar[r]^-{\cdot 2}&{M}\ar[r]^-{\cdot 2}&{M}\ar[r]&{\Sigma M}}.$$
Hence, it is obvious that in the mixed-characteristic case we are dealing with such exotic triangulated categories. Actually, the examples considered in \cite{MSS} are exactly the categories of finitely generated projective modules over commutative rings $R$, that satisfy the conditions for the mixed-characteristic case.\\

\end{document}